\documentclass[11pt]{article}

\usepackage{amssymb, amsmath, fullpage, amsthm}
\usepackage{mathrsfs}

\newtheorem{theorem}{Theorem}[section]

\newtheorem{corollary}[theorem]{Corollary}

 \numberwithin{equation}{section}

\newcommand{\hz}{\widehat{0}}
\newcommand{\ho}{\widehat{1}}

\newcommand{\av}{{\bf a}}
\newcommand{\bv}{{\bf b}}
\newcommand{\cv}{{\bf c}}
\newcommand{\dv}{{\bf d}}
\newcommand{\ab}{\av\bv}
\newcommand{\cd}{\cv\dv}
\newcommand{\zab}{{\mathbb Z}\langle\av,\bv\rangle}

\newcommand{\rv}{{\bf r}}

\newcommand{\tensor}{\otimes}

\newcommand{\coveredby}{\prec}

\DeclareMathOperator{\Pyr}{Pyr}
\DeclareMathOperator{\Bipyr}{Bipyr}

\DeclareMathOperator{\des}{des}
\DeclareMathOperator{\maj}{maj}

\DeclareMathOperator{\Des}{Des}
\DeclareMathOperator{\ps}{ps}
\DeclareMathOperator{\Comp}{Comp}
\DeclareMathOperator{\co}{co}
\DeclareMathOperator{\QSym}{QSym}

\newcommand{\JH}{JH}

\newcommand{\Zzz}{{\mathbb Z}}

\newcommand{\SSSS}{\mathfrak{S}}

\newcommand{\TTTT}[1]{B_{#1} \cup \{\ho\}}

\newcommand{\Gaussian}[2]{\genfrac{[}{]}{0pt}{}{#1}{#2}}

\begin{document}

\title{A Poset View of the Major Index}

\author{Richard Ehrenborg\thanks{Corresponding author:
Department of Mathematics,
University of Kentucky,
Lexington, KY 40506-0027,
USA,
{\tt jrge@ms.uky.edu},
phone +1 (859) 257-4090,
fax +1 (859)  257-4078.}
and Margaret Readdy\thanks{Department of Mathematics,
University of Kentucky,
Lexington, KY 40506-0027,
USA,
{\tt readdy@ms.uky.edu}.}
}

\date{}

\maketitle

\begin{abstract}
We introduce the Major MacMahon map from
$\zab$ to $\Zzz[q]$,
and
show how this map interacts
with the pyramid and bipyramid operators.
When the Major MacMahon map is applied
to the $\ab$-index of
a simplicial poset, it yields the $q$-analogue of
$n!$ times the $h$-polynomial of the poset.
Applying the map to the Boolean algebra gives
the distribution of the major index on the symmetric group,
a seminal result due to MacMahon.
Similarly, when applied to the cross-polytope we obtain
the distribution of one of the major indexes on
signed permutations due to Reiner.

\vspace*{2 mm}


\vspace*{2 mm}

\end{abstract}

\section{Introduction}

One hundred and one years ago in 1913
Major Percy Alexander MacMahon~\cite{MacMahon_1}
(see also his collected works~\cite{MacMahon_collected})
introduced the major index of
a permutation $\pi = \pi_{1} \pi_{2} \cdots \pi_{n}$
of the multiset
$M = \{1^{\alpha_{1}}, 2^{\alpha_{2}}, \ldots, k^{\alpha_{k}}\}$
of size $n$
to be the sum of the elements
of its descent set, that is,
$$   \maj(\pi) = \sum_{\pi_{i} > \pi_{i+1}} i   .  $$
He showed that the distribution of this permutation statistic 
is given by the $q$-analogue of the multinomial 
Gaussian coefficient, that is,
the following identity holds:
\begin{equation}
    \sum_{\pi} q^{\maj(\pi)}
    =
      \frac{[n]!}{[\alpha_{1}]! \cdot [\alpha_{2}]! \cdots [\alpha_{k}]!}
    =
       \Gaussian{n}{\alpha}  ,
\label{equation_n!}
\end{equation}
where $\pi$ ranges over all permutations of the multiset $M$
and $\alpha$ is the composition
$(\alpha_{1}, \alpha_{2}, \ldots, \alpha_{k})$.
Here $[n]! = [n] \cdot [n-1] \cdots [1]$
denotes the $q$-analogue of $n!$,
where $[n] = 1 + q + \cdots + q^{n-1}$.

Many properties of the {\em descent set} of a permutation
$\pi$, that is, $\Des(\pi) = \{i \: : \: \pi_{i} > \pi_{i+1}\}$,
have been studied by encoding the set by its $\ab$-word;
see for instance~\cite{Ehrenborg_Readdy_r_cubical,Readdy}.
For a multiset permutation $\pi \in \SSSS_{M}$
the {\em $\ab$-word} is given by $u(\pi) = u_{1} u_{2} \cdots u_{n-1}$,
where 
$u_{i} = \bv$ if $\pi_{i} > \pi_{i+1}$ 
and
$u_{i} = \av$ otherwise.

Inspired by this definition, we introduce the
{\em Major MacMahon map} $\Theta$ on 
the ring $\zab$ of non-commutative polynomials in the variables
$\av$ and $\bv$ to the ring $\Zzz[q]$ of polynomials in the variable~$q$,
by
$$ \Theta(w) = \prod_{i \: : \: u_{i} = \bv} q^{i}  , $$
for a monomial $w = u_{1} u_{2} \cdots u_{n}$
and extend $\Theta$ to all of $\zab$ by linearity.
In short, the map~$\Theta$ sends each variable $\av$ to
$1$ and the variables $\bv$ to $q$ to the power of its position,
read from left to right.
A Swedish example is $\Theta(\av\bv\bv\av) = q^{5}$.

\section{Chain enumeration and products of posets}

Let $P$ be a graded poset of rank $n+1$
with minimal element $\hz$,
maximal element $\ho$
and
rank function~$\rho$.
Let the rank difference be defined by
$\rho(x,y) = \rho(y) - \rho(x)$.
The {\em flag $f$-vector} entry $f_{S}$,
for $S = \{s_{1} < s_{2} < \cdots < s_{k}\}$ a subset $\{1,2, \ldots, n\}$,
is the number of chains
$c = \{\hz = x_{0} < x_{1} < x_{2} < \cdots < x_{k+1} = \ho\}$
such that
the rank of the element $x_{i}$ is $s_{i}$,
that is, $\rho(x_{i}) = s_{i}$ for $1 \leq i \leq k$.
The {\em flag $h$-vector} is defined by the invertible relation
$$   h_{S} 
    =
       \sum_{T \subseteq S} (-1)^{|S-T|} \cdot f_{T}  .  $$
For a subset $S$ of $\{1,2, \ldots, n\}$ define two
$\ab$-polynomials of degree $n$ 
by $u_{S} = u_{1} u_{2} \cdots u_{n}$
and $v_{S} = v_{1} v_{2} \cdots v_{n}$ by
$$
   u_{i} = \begin{cases} \av & \text{ if $i \notin S$,} \\
                                     \bv & \text{ if $i \in S$,}
                  \end{cases}
\:\:\:\: \text{ and } \:\:\:\:
   v_{i} = \begin{cases} \av-\bv & \text{ if $i \notin S$,} \\
                                     \bv & \text{ if $i \in S$.}
                  \end{cases}
$$
The {\em $\ab$-index} of the poset $P$ is defined by
the two equivalent expressions:
$$   \Psi(P)
    =
       \sum_{S} f_{S} \cdot v_{S}
    =
       \sum_{S} h_{S} \cdot u_{S}   , $$
where the two sums range over all subsets $S$ of $\{1,2, \ldots, n\}$.
For more details on the $\ab$-index,
see~\cite{Ehrenborg_Readdy} or
the book~\cite[Section~3.17]{EC1}.

Recall that a graded poset $P$ is {\em Eulerian}
if every non-trivial interval has the same number
of elements of even rank as odd rank.
Equivalently, a poset is Eulerian if its M\"obius function
satisfies $\mu(x,y) = (-1)^{\rho(x,y)}$
for all $x \leq y$ in $P$.
When the graded poset $P$ is Eulerian
then the $\ab$-index $\Psi(P)$ can be written
in terms of the non-commuting variables
$\cv = \av+\bv$ and $\dv = \av\bv + \bv\av$
and it is called the $\cd$-index;
see~\cite{Bayer_Klapper}.
For an $n$-dimensional convex polytope $V$ its face lattice 
${\mathscr L}(V)$ is
an Eulerian poset of rank~$n+1$.
In this case we write $\Psi(V)$ for the $\ab$-index ($\cd$-index)
instead of the cumbersome~$\Psi({\mathscr L}(V))$.

There are also two products on graded posets that we will study.
The first is the {\em Cartesian product}, defined by
$P \times Q = \{(x,y) \: : \: x \in P, y \in Q\}$ with the
order relation $(x,y) \leq_{P \times Q} (z,w)$
if $x \leq_{P} z$ and $y \leq_{Q} w$.
Note that the rank of the Cartesian product of
two graded posets of ranks $m$ and $n$ is $m+n$.
As a special case we define
$\Pyr(P) = P \times B_{1}$, where $B_{1}$ is the Boolean
algebra of rank~$1$. The geometric reason
for the notation $\Pyr$
is that
this operation corresponds to the geometric operation
of taking the pyramid of a polytope, that is,
${\mathscr L}(\Pyr(V)) = \Pyr({\mathscr L}(V))$
for a polytope $V$.

The second product is the {\em dual diamond product},
defined by
$$  P \diamond^{*} Q
  =
  (P - \{\ho_{P}\}) \times (Q - \{\ho_{Q}\}) \cup \{\ho\}  .  $$
The rank of the product
$P \diamond^{*} Q$
is the sum of the ranks of $P$ and $Q$ minus one.
This is the dual to the diamond product $\diamond$ defined by
removing the minimal elements of the posets,
taking the Cartesian product
and then adjoining a new minimal element. The product $\diamond$
behaves well with the quasi-symmetric functions of type $B$.
(See Sections~\ref{section_Cartesian}
and~\ref{section_diamond}.)
However, we will dualize our presentation and keep working with
the product~$\diamond^{*}$.

Yet again, we have an important special case. We define
$\Bipyr(P) = P \diamond^{*} B_{2}$.
The geometric motivation 
is the connection to the bipyramid of a polytope, that is,
${\mathscr L}(\Bipyr(V)) = \Bipyr({\mathscr L}(V))$
for a polytope $V$.

\section{Pyramids and bipyramids}

Define on the ring $\zab$
of non-commutative polynomials in the variables $\av$ and $\bv$
the two derivations~$G$ and $D$ by
$$ \begin{array}{c l}
G(1) = 0, & G(\av) = \bv\av, \:\:\: G(\bv) = \av\bv , \\
D(1) = 0, & D(\av) = D(\bv) = \av\bv + \bv\av .
\end{array} $$
Extend these two derivations to all of $\zab$ by linearity.
The {\em pyramid} and the {\em bipyramid operators} are given
by
$$
\Pyr(w) = G(w) + w \cdot \cv 
\:\:\:\: \text{ and } \:\:\:\:
\Bipyr(w) = D(w) + \cv \cdot w .
$$
These two operators are suitably named,
since for a graded poset $P$ we have
$$   \Psi(\Pyr(P)) = \Pyr(\Psi(P)) 
\:\:\:\: \text{ and } \:\:\:\:
     \Psi(\Bipyr(P)) = \Bipyr(\Psi(P)) . $$
For further details, see~\cite{Ehrenborg_Readdy}.

\begin{theorem}
The Major MacMahon map $\Theta$ interacts with
right multiplication by $\cv$,
the derivation~$G$,
the pyramid and the bipyramid operators as follows:
\begin{align}
   \Theta(w \cdot \cv) & =  (1 + q^{n+1}) \cdot \Theta(w) ,
\label{equation_c} \\
   \Theta(G(w)) & =  q \cdot [n] \cdot \Theta(w) ,
\label{equation_G} \\
   \Theta(\Pyr(w)) & =  [n+2] \cdot \Theta(w) , 
\label{equation_pyramid} \\
   \Theta(\Bipyr(w)) & =  [2] \cdot [n+1] \cdot \Theta(w) , 
\label{equation_bipyramid}
\end{align}
where $w$ is a homogeneous $\ab$-polynomial of degree $n$.
\end{theorem}
\begin{proof}
It is enough to prove the four identities for an $\ab$-monomial $w$
of degree $n$.
Directly we have that
$\Theta(w \cdot \av) = \Theta(w)$
and
$\Theta(w \cdot \bv) = q^{n+1} \cdot \Theta(w)$.
Adding these two identities yields
equation~\eqref{equation_c}.

Assume that $w$ consists of $k$ $\bv$'s.
We label the $n$ letters of $w$ as follows:
The $k$ $\bv$'s are labeled $1$ through $k$ reading from
right to left, whereas the $n-k$ $\av$'s are labeled
$k+1$ through~$n$ reading left to right. As an example, the word
$w = \av\av\bv\av\bv\bv\av$ is written
as $w_{4} w_{5} w_{3} w_{6} w_{2} w_{1} w_{7}$.

Identity~\eqref{equation_G}
is a consequence of the following claim.
Applying the derivation $G$ only to the letter~$w_{i}$ and
then applying the Major MacMahon map yields $q^{i} \cdot \Theta(w)$,
that is,
\begin{equation}
 \Theta(u \cdot G(w_{i}) \cdot v) = q^{i} \cdot \Theta(u \cdot w_{i} \cdot v) , 
\label{equation_i}
\end{equation}
where $w$ is factored as $u \cdot w_{i} \cdot v$.
To see this, first consider when $1 \leq i \leq k$.
There are $i$ $\bv$'s to the right of $w_{i}$
including $w_{i}$ itself. They each are shifted one step to the right when
replacing $w_{i} = \bv$ with $G(\bv) = \av\bv$ and hence
we gain a factor of $q^{i}$.
The second case is when $k+1 \leq i \leq n$.
Then $w_{i}$ is an $\av$ and is replaced by $\bv\av$
under the derivation $G$.
Assume that there are $j$ $\bv$'s to the right of $w_{i}$.
When these $j$ $\bv$'s are shifted one step to the right they contribute
a factor of $q^{j}$. We also create a new~$\bv$. 
It has $i-k-1$ $\av$'s to the left and $k-j$ $\bv$'s to the left.
Hence the position of the new $\bv$ is
$(i-k-1) + (k-j) + 1 = i-j$ and thus its contribution is $q^{i-j}$.
Again the factor is given by $q^{j} \cdot q^{i-j} = q^{i}$,
proving the claim.
Now by summing over these $n$ cases,
identity~\eqref{equation_G} follows.
Identity~\eqref{equation_pyramid} 
is the sum of identities~\eqref{equation_c} 
and~\eqref{equation_G}.

To prove identity~\eqref{equation_bipyramid},
we use a different labeling of
the monomial $w$. This time label the $k$ $\bv$'s with
the subscripts $0$ through $k-1$,
rather than $1$ through $k$. That is, in our example
$w = \av\av\bv\av\bv\bv\av$ is now labeled as
$w_{4} w_{5} w_{2} w_{6} w_{1} w_{0} w_{7}$.
We claim that for
$w = u \cdot w_{i} \cdot v$ we have that
$$ \Theta(u \cdot D(w_{i}) \cdot v) = q^{i} \cdot [2] \cdot \Theta(w) . $$
The first case is $0 \leq i \leq k-1$. Then $w_{i} = \bv$
has $i$ $\bv$'s to its right. Thus when replacing $\bv$ with~$\bv\av$
there are $i$ $\bv$'s that are shifted one step, giving the factor $q^{i}$.
Similarly, when replacing $w_{i}$ with $\av\bv$, there are
$i+1$ $\bv$'s that are shifted one step, giving the factor $q^{i+1}$.
The sum of the two factors is $q^{i} \cdot [2]$.
The second case is $k+1 \leq i \leq n$.
It is as the second case above when replacing $w_{i}$ with~$\bv\av$,
yielding the factor $q^{i}$. When replacing $w_{i}$ with~$\av\bv$
there is one more shift, giving $q^{i+1}$. Adding these two subcases
completes the proof of the claim.

It is straightforward to observe that 
$$ \Theta(\cv \cdot w) = q^{k} \cdot [2] \cdot \Theta(w) . $$
Calling this the case $i=k$, the identity~\eqref{equation_bipyramid}
follows by summing the $n+1$ cases $0 \leq i \leq n$.
\end{proof}

Iterating
equations~\eqref{equation_pyramid}
and~\eqref{equation_bipyramid}
we obtain
that the Major MacMahon map of the $\ab$-index of
the $n$-dimensional simplex $\Delta_{n}$
and
the $n$-dimensional cross-polytope $C^{*}_{n}$.
\begin{corollary}
The $n$-dimensional simplex $\Delta_{n}$
and the $n$-dimensional cross-polytope $C^{*}_{n}$
satisfy
\begin{align*}
  \Theta(\Psi(\Delta_{n})) & = [n+1]!   ,                   \\
  \Theta(\Psi(C^{*}_{n}))  & = [2]^{n} \cdot [n]!   .
\end{align*}
\end{corollary}

\section{Simplicial posets}

A graded poset $P$ is {\em simplicial} if all of its lower order intervals
are Boolean, that is, for all elements $x < \ho$
the interval $[\hz,x]$ is isomorphic to
the Boolean algebra $B_{\rho(x)}$.
It is well-known that
all the flag information of a simplicial poset of rank $n+1$ is contained
in the $f$-vector $(f_{0}, f_{1}, \ldots, f_{n})$,
where
$f_{0} = 1$ and $f_{i} = f_{\{i\}}$ for $1 \leq i \leq n$.
The $h$-vector, equivalently,
the $h$-polynomial
$h(P)
  =
h_{0} + h_{1} \cdot q + \cdots + h_{n} \cdot q^{n}$
of a simplicial poset $P$,
is defined by the polynomial relation
$$
  h(q)
 = 
  \sum_{i=0}^{n} f_{i} \cdot q^{i} \cdot (1-q)^{n-i}  .  $$
See for instance~\cite[Section~8.3]{Ziegler}.
The $h$-polynomial and the bipyramid operation
interact as follows:
$$
h(\Bipyr(P)) = (1+q) \cdot h(P) .
$$

We can now evaluate the Major MacMahon map on the
$\ab$-index of a simplicial poset.
\begin{theorem}
For a simplicial poset $P$ of rank $n+1$
the following identity holds:
\begin{equation}
  \Theta(\Psi(P)) = [n]! \cdot h(P)   . 
\label{equation_simplicial_poset}
\end{equation}
\end{theorem}
\begin{proof}
Let $\TTTT{n}$ denote the Boolean algebra $B_{n}$
with a new maximal element added.
Note that $\TTTT{n}$ is indeed a simplicial poset
and its $h$-polynomial is $1$. Furthermore,
equation~\eqref{equation_simplicial_poset}
holds for $\TTTT{n}$ since
$$ \Theta(\Psi(\TTTT{n})) 
   =
     \Theta(\Psi(B_{n}) \cdot \av) 
   =
     \Theta(\Psi(B_{n}))
   =
     [n]!
   =
     [n]! \cdot h(\TTTT{n})  .  $$
Also, if~\eqref{equation_simplicial_poset}
holds for a poset $P$ then it also holds for $\Bipyr(P)$,
since we have
$$     \Theta(\Psi(\Bipyr(P)))
      =
         [2] \cdot [n+1] \cdot \Theta(\Psi(P))
      =
         [2] \cdot [n+1] \cdot [n]! \cdot h(P)
      =
         [n+1]! \cdot h(\Bipyr(P)) .
$$

Observe that both sides of~\eqref{equation_simplicial_poset}
are linear in the $h$-polynomial.
Hence to prove it for any simplicial poset $P$
it is enough to prove it for a basis of the span
of all simplicial posets of rank~$n+1$.
Such a basis is given by the posets
$$   {\mathcal B}_{n}
    =     
       \left\{ \Bipyr^{i}(\TTTT{n-i}) \right\}
                        _{0 \leq i \leq n}     . $$ 
This is a basis since the polynomials
$h(\Bipyr^{i}(\TTTT{n-i})) = (1+q)^{i}$,
for $0 \leq i \leq n$, are a basis for
polynomials in the variable $q$ of degree at most $n$.

Finally, since every element in the basis is built up
by iterating bipyramids of the posets $\TTTT{n}$, the theorem holds
for all simplicial posets.
\end{proof}

Observe
that the poset $\Bipyr^{i}(\TTTT{n-i})$
is the face lattice of the simplicial complex
consisting of the $2^{i}$ facets of the $n$-dimensional
cross-polytope in the cone $x_{1}, \ldots, x_{n-i} \geq 0$.

For an Eulerian simplicial poset $P$, the
$h$-vector is symmetric, that is, $h_{i} = h_{n-i}$.
In other words, the $h$-polynomial is palindromic.
Stanley~\cite{Stanley} introduced the
{\em simplicial shelling components},
that is, 
the $\cd$-polynomials~$\check{\Phi}_{n,i}$
such that the $\cd$-index of
an Eulerian simplicial poset $P$ of rank $n+1$ is given by
\begin{equation}
  \Psi(P)
    =
  \sum_{i=0}^{n} h_{i} \cdot \check{\Phi}_{n,i}   . 
\label{equation_simplicial}
\end{equation}
These $\cd$-polynomials satisfy the recursion
$\check{\Phi}_{n,0} = \Psi(B_{n}) \cdot \cv$
and
$\check{\Phi}_{n,i} = G(\check{\Phi}_{n-1,i-1})$;
see~\cite[Section~8]{Ehrenborg_Readdy}.
The Major MacMahon map of these polynomials is
described by the next result.
\begin{corollary}
The Major MacMahon map of
the simplicial shelling components is given by
$$   \Theta(\check{\Phi}_{n,i})
     =
       q^{i} \cdot [2(n-i)] \cdot [n-1]!  .       $$
\end{corollary}
\begin{proof}
When $i=0$ we have
$\Theta(\check{\Phi}_{n,0})
   =
\Theta(\Psi(B_{n}) \cdot \cv)
   =
(1+q^{n}) \cdot [n]!
   =
[2n] \cdot [n-1]!$.
Also when $i \geq 1$ we obtain
$\Theta(\check{\Phi}_{n,i})
   =
\Theta(G(\check{\Phi}_{n-1,i-1}))
   =
q \cdot [n-1] \cdot \Theta(\check{\Phi}_{n-1,i-1})
   =
q^{i} \cdot [2(n-i)] \cdot [n-1]!$.
\end{proof}

We end with the following observation.
\begin{theorem}
For an Eulerian poset $P$ of rank $n+1$,
the polynomial $[2]^{\lceil n/2 \rceil}$
divides $\Theta(\Psi(P))$.
\end{theorem}
\begin{proof}
It is enough to show this result for a $\cd$-monomial $w$
of degree $n$.
A $\cv$ in an odd position~$i$ of~$w$ yields a factor of $1+q^{i}$.
A $\dv$ that covers an odd position $i$ of~$w$ yields either
$q^{i-1}+q^{i}$ or $q^{i}+q^{i+1}$.
Each of these polynomials contributes a factor of $1+q$.
The result follows since there are 
$\lceil n/2 \rceil$ odd positions.
\end{proof}

\section{The Cartesian product of posets}
\label{section_Cartesian}

We now study how the Major MacMahon map behaves
under the Cartesian product.
Recall that for a graded poset~$P$ the $\ab$-index $\Psi(P)$ encodes
the flag $f$-vector information of the poset $P$.
There is another encoding of this information as a quasi-symmetric
function. For further information about quasi-symmetric functions,
see~\cite[Section~7.19]{EC2}.

A composition $\alpha$ of $n$ is a list of positive integers
$(\alpha_{1}, \alpha_{2}, \ldots, \alpha_{k})$ such that 
$\alpha_{1} + \alpha_{2} + \cdots + \alpha_{k} = n$.
Let $\Comp(n)$ denote the set
of compositions of $n$. There are three natural bijections
between $\ab$-monomials $u$ of degree $n$, subsets $S$
of the set $\{1,2,\ldots,n\}$ and compositions of $n+1$.
Given a composition 
$\alpha \in \Comp_{n+1}$
we have the subset $S_{\alpha}$,
the $\ab$-monomial $u_{\alpha}$
and the $\ab$-polynomial $v_{\alpha}$
defined by
\begin{align*}
S_{\alpha}
    &   = 
    \{\alpha_{1}, \alpha_{1} + \alpha_{2}, \ldots,
                   \alpha_{1} + \cdots + \alpha_{k-1}\} , \\
u_{\alpha}
    &  =
\av^{\alpha_{1}-1}  \cdot
\bv \cdot
\av^{\alpha_{2}-1}  \cdot
\bv \cdots
\bv \cdot
\av^{\alpha_{k}-1} , \\
v_{\alpha}
    &  =
(\av-\bv)^{\alpha_{1}-1}  \cdot
\bv \cdot
(\av-\bv)^{\alpha_{2}-1}  \cdot
\bv \cdots
\bv \cdot
(\av-\bv)^{\alpha_{k}-1} .
\end{align*}
For $S$ a subset of $\{1,2, \ldots, n\}$
let $\co(S)$ denote associated composition.

The {\em monomial quasi-symmetric function} $M_{\alpha}$
is defined as the sum
$$   M_{\alpha}
    =
      \sum_{i_{1} < i_{2} < \cdots < i_{k}}
            t_{i_{1}}^{\alpha_{1}} \cdot
            t_{i_{2}}^{\alpha_{2}} \cdots
            t_{i_{k}}^{\alpha_{k}}    .  $$
A second basis is given by
the {\em fundamental quasi-symmetric function} $L_{\alpha}$
defined as
$$   L_{\alpha}
    =
      \sum_{S_{\alpha} \subseteq T \subseteq \{1,2, \ldots, n\}}
             M_{\co(T)}   .  $$

Following~\cite{Ehrenborg_Readdy_Tchebyshev_transform}
define an injective linear map
$\gamma : \zab \longrightarrow \QSym$ by
\begin{equation*}
   \gamma\left( v_{\alpha} \right)
  =
  M_{\alpha} ,
\label{equation_gamma}
\end{equation*}
for a composition $\alpha$ of $n \geq 1$.
The image of $\gamma$ is all quasi-symmetric functions
without constant term. Moreover, the image
of the $\ab$-monomial $u_{\alpha}$ under $\gamma$
is the fundamental
quasi-symmetric function
$L_{\alpha}$, that is,
$$   \gamma(u_{\alpha}) = L_{\alpha} .  $$

Another way to encode the flag vectors of a poset $P$
is by the {\em quasi-symmetric function} of the poset.
It is quickly defined as $F(P) = \gamma(\Psi(P))$.
A more poset-oriented definition is
the following limit of sums over multichains:
$$   F(P)
   =
       \lim_{k \longrightarrow \infty}
\sum_{\hz = x_{0} \leq x_{1} \leq \cdots \leq x_{k} = \ho}
            t_{1}^{\rho(x_{0},x_{1})}
               \cdot
            t_{2}^{\rho(x_{1},x_{2})}
               \cdots
            t_{k}^{\rho(x_{k-1},x_{k})}   .   $$
For more on the quasi-symmetric function of a poset,
see~\cite{Ehrenborg_Hopf}.

The {\em stable principal specialization} of a quasi-symmetric
function is the substitution
$\ps(f) = f(1,q,q^{2}, \ldots)$.
Note that this is a homeomorphism, that is,
$\ps(f \cdot g) = \ps(f) \cdot \ps(g)$.

For a composition $\alpha = (\alpha_{1}, \alpha_{2}, \ldots, \alpha_{k})$
let $\alpha^{*}$ denote the reverse composition, that is,
$\alpha^{*} = (\alpha_{k}, \ldots, \alpha_{2}, \alpha_{1})$.
This involution extends to an anti-automorphism on $\QSym$
by $M_{\alpha}^{*} \longmapsto M_{\alpha^{*}}$.
Define $\ps^{*}$ by the relation
$\ps^{*}(f) = \ps(f^{*})$. Informally speaking, this corresponds to
the substitution $\ps^{*}(f) = f(\ldots,q^{2},q,1)$.

\begin{theorem}
For a homogeneous $\ab$-polynomial $w$ of degree $n-1$
the Major MacMahon map is given by
\begin{equation}
\Theta(w)  =  (1-q)^{n} \cdot [n]! \cdot \ps^{*}(\gamma(w)) . 
\label{equation_ps_1}
\end{equation}
For a poset $P$ of rank $n$ this identity is
\begin{equation}
\Theta(\Psi(P))  =  (1-q)^{n} \cdot [n]! \cdot \ps^{*}(F(P)) . 
\label{equation_ps_2}
\end{equation}
\end{theorem}
\begin{proof}
It is enough to prove identity~\eqref{equation_ps_1}
for an $\ab$-monomial $w$ of degree $n-1$.
Let $\alpha$ be the composition of $n$ corresponding
to the reverse monomial $w^{*}$.
Furthermore, let $e(\alpha)$ be the sum
$\sum_{i \in S_{\alpha}} (n-i)$.
Note that $e(\alpha)$ is in fact the sum $\sum_{i \in S} i$,
where $S$ is the subset associated with the $\ab$-monomial $w$.
That is, we have $q^{e(\alpha)} = \Theta(w)$.
Equation~\eqref{equation_ps_1} follows
from Lemma~7.19.10 in~\cite{EC2}.
By applying the first identity to $\Psi(P)$,
we obtain identity~\eqref{equation_ps_2}.
\end{proof}

Since the quasi-symmetric function is multiplicative under
the Cartesian product, we have the next result.
\begin{theorem}
For two posets $P$ and $Q$ of ranks $m$, respectively $n$,
the following identity holds:
\begin{equation}
\Theta(\Psi(P \times Q))
  =
\Gaussian{m+n}{n}
  \cdot
\Theta(\Psi(P))
  \cdot
\Theta(\Psi(Q)) .
\end{equation}
\label{theorem_Cartesian_product}
\end{theorem}
\begin{proof}
The proof is a direct verification as follows:
\begin{align*}
\Theta(\Psi(P \times Q))
   &  =
 (1-q)^{m+n} \cdot [m+n]! \cdot \ps(F(P^{*} \times Q^{*}))  \\
   &  =
 \Gaussian{m+n}{m} \cdot (1-q)^{m+n} \cdot [m]! \cdot [n]!
 \cdot \ps(F(P^{*})) \cdot \ps(F(Q^{*}))  \\
   &  =
 \Gaussian{m+n}{m} \cdot \Theta(\Psi(P))
 \cdot \Theta(\Psi(Q)) .
\qedhere
\end{align*}
\end{proof}

\section{The dual diamond product}
\label{section_diamond}

Define the {\em quasi-symmetric function of type $B^{*}$} of
a graded poset $P$
to be the expression
$$    F_{B^{*}}(P)
     =
        \sum_{\hz \leq x < \ho} F([\hz,x]) \cdot s^{\rho(x,\ho)-1} .  $$
This is an element of the algebra $\QSym \tensor \Zzz[s]$
which we view as the quasi-symmetric functions of type $B^{*}$.
We view $\QSym_{B^{*}}$ as a subalgebra of
$\Zzz[t_{1},t_{2}, \ldots; s]$,
which is quasi-symmetric in the variables $t_{1}, t_{2}, \ldots$.
For instance, a basis for
$\QSym_{B^{*}}$ is given by $M_{\alpha} \cdot s^{i}$
where $\alpha$ ranges over all compositions and $i$ over all
non-negative integers.
Similar to the map
$\gamma: \zab \longrightarrow \QSym$,
we define
$\gamma_{B^{*}}: \zab \longrightarrow \QSym_{B^{*}}$ by
$$
\gamma_{B^{*}}\left(
(\av-\bv)^{\alpha_{1}-1}
\cdot \bv \cdot 
(\av-\bv)^{\alpha_{2}-1}
\cdot \bv \cdots \bv \cdot
(\av-\bv)^{\alpha_{k}-1}
\cdot \bv \cdot 
(\av-\bv)^{p}
\right)
 =
M_{\alpha} \cdot s^{p} ,
$$
where $\alpha$ is the composition
$(\alpha_{1}, \alpha_{2}, \ldots, \alpha_{k})$. 
Similar to the relation
$\gamma(\Psi(P)) = F(P)$, we have
$$ \gamma_{B^{*}}(\Psi(P))   =   F_{B^{*}}(P)   .  $$
Furthermore, the type $B^{*}$ quasi-symmetric function $F_{B^{*}}$
is multiplicative respect to the product~$\diamond^{*}$,
that is,
$F_{B^{*}}(P \diamond^{*} Q) = F_{B^{*}}(P) \cdot F_{B^{*}}(Q)$;
see~\cite[Theorem~13.3]{Ehrenborg_Readdy_Tchebyshev_transform}.

Let $f$ be a homogeneous quasi-symmetric function such that
$f \cdot s^{j}$ is a quasi-symmetric function of type $B^{*}$.
Define the {\em stable principal specialization} of
the quasi-symmetric function $f \cdot s^{j}$
of type~$B^{*}$ to be
$\ps_{B^{*}}(f \cdot s^{j}) = q^{\deg(f)} \cdot \ps^{*}(f)$,
where $\ps^{*}(f) = \ps(f^{*})$.
This is the substitution $s=1$, $t_{k} = q$, $t_{k-1} = q^{2}$, \ldots
as $k$ tends to infinity,
since $f(\ldots, q^{3},q^{2},q) = q^{\deg(f)} \cdot f(\ldots,q^{2},q,1)$.
Especially, for a graded poset $P$ we have
\begin{equation}
       \ps_{B^{*}}(F_{B^{*}}(P))
     =
        \sum_{\hz \leq x < \ho} q^{\rho(x)} \cdot \ps^{*}(F([\hz,x])) . 
\label{equation_ps_star}
\end{equation}

\begin{theorem}
For a graded poset $P$ of rank $n+1$ the relationship
between the Major MacMahon map and
the stable principal specialization of type $B^{*}$ 
is given by
\begin{equation}
\Theta(\Psi(P))  =  (1-q)^{n} \cdot [n]! \cdot \ps_{B^{*}}(F_{B^{*}}(P^{*})) . 
\label{equation_B_ps_2}
\end{equation}
Especially, for a homogeneous $\ab$-polynomial $w$ of degree $n$
the Major MacMahon map is given by
\begin{equation}
\Theta(w)  =  (1-q)^{n} \cdot [n]! \cdot \ps_{B^{*}}(\gamma_{B^{*}}(w^{*})) . 
\label{equation_B_ps_1}
\end{equation}
\end{theorem}
\begin{proof}
For the poset $P$ we have
\begin{align*}
\ps^{*}(F(P))
& = 
\lim_{k \rightarrow \infty}
\sum_{\hz = x_{0} \leq x_{1} \leq \cdots \leq x_{k} = \ho}
            \left(q^{k-1}\right)^{\rho(x_{0},x_{1})}
               \cdots
            \left(q^{2}\right)^{\rho(x_{k-3},x_{k-2})}
               \cdot
            q^{\rho(x_{k-2},x_{k-1})}
               \cdot
            1^{\rho(x_{k-1},x_{k})} \\
& = 
\lim_{k \rightarrow \infty}
\sum_{\hz = x_{0} \leq x_{1} \leq \cdots \leq x_{k} = \ho}
            q^{\rho(x_{k-1})}
         \cdot
            \left(q^{k-2}\right)^{\rho(x_{0},x_{1})}
               \cdots
            q^{\rho(x_{k-3},x_{k-2})}
               \cdot
            1^{\rho(x_{k-2},x_{k-1})} \\
& = 
\sum_{\hz \leq x \leq \ho} 
            q^{\rho(x)}
         \cdot
\ps^{*}(F([\hz,x])) \\
& = 
\sum_{\hz \leq x < \ho} 
            q^{\rho(x)}
         \cdot
\ps^{*}(F([\hz,x])) 
+
            q^{n+1}
         \cdot
\ps^{*}(F(P)) .
\end{align*}
Rearranging terms yields
\begin{align*}
\sum_{\hz \leq x < \ho}  q^{\rho(x)} \cdot \ps^{*}(F([\hz,x])) 
  & = 
(1- q^{n+1}) \cdot  \ps^{*}(F(P)) \\
  & = 
(1 - q^{n+1}) \cdot  \ps(F(P^{*})) \\
  & = 
(1 - q^{n+1}) \cdot  
\frac{\Theta(\Psi(P))}{(1-q)^{n+1} \cdot [n+1]!} \\
  & = 
\frac{\Theta(\Psi(P))}{(1-q)^{n} \cdot [n]!} .
\end{align*}
Combining the last identity with~\eqref{equation_ps_star}
yields the desired result.
\end{proof}

\begin{theorem}
For two graded posets $P$ and $Q$
of ranks $m+1$, respectively $n+1$,
the identity holds:
\begin{equation}
\Theta(\Psi(P \diamond^{*} Q))
  =
\Gaussian{m+n}{n}
  \cdot
\Theta(\Psi(P))
  \cdot
\Theta(\Psi(Q)) .
\end{equation}
\end{theorem}
\begin{proof}
The proof is a direct verification as follows:
\begin{align*}
\Theta(\Psi(P \diamond^{*} Q))
   &  =
 (1-q)^{m+n} \cdot [m+n]! \cdot \ps_{B^{*}}(F_{B^{*}}(P^{*} \diamond^{*} Q^{*}))  \\
   &  =
 \Gaussian{m+n}{m} \cdot (1-q)^{m+n} \cdot [m]! \cdot [n]!
 \cdot \ps_{B^{*}}(F_{B^{*}}(P^{*})) \cdot \ps_{B^{*}}(F_{B^{*}}(Q^{*}))  \\
   &  =
 \Gaussian{m+n}{m} \cdot \Theta(\Psi(P))
 \cdot \Theta(\Psi(Q)) .
\qedhere
\end{align*}
\end{proof}

\section{Permutations}
\label{section_permutations}

One connection between permutations and posets is via
the concept of $R$-labelings. For more details,
see~\cite[Section~3.14]{EC1}.
Let ${\mathcal E}(P)$ be the set of all cover relations of $P$,
that is, 
${\mathcal E}(P) = \{(x,y) \in P^{2} \: : \: x \coveredby y\}$.
A graded poset $P$ has an {\em $R$-labeling} if there is a map
$\lambda : {\mathcal E}(P) \longrightarrow \Lambda$, where $\Lambda$
is a linearly ordered set, such that in every interval $[x,y]$
in $P$ there is a unique maximal chain
$c = \{x = x_{0} \coveredby x_{1} \coveredby \cdots \coveredby x_{k} = y\}$
such that
$\lambda(x_{0},x_{1}) \leq \lambda(x_{1},x_{2}) \leq \cdots 
\cdots \leq \lambda(x_{k-1},x_{k})$.

For a maximal chain $c$ in the poset $P$ of rank $n$,
let $\lambda(c)$ denote the list
$(\lambda(x_{0},x_{1})$, $\lambda(x_{1},x_{2})$, \ldots,
$\lambda(x_{k-1},x_{k}))$.
The {\em Jordan--H\"older set} of $P$, denoted by $\JH(P)$,
is the set of all the lists
$\lambda(c)$ where $c$ ranges over all maximal chains of $P$.
The descent set of a list of labels $\lambda(c)$
is the set of positions where there are descents in the list.
Similarly, we define the descent word of $\lambda(c)$
to be $u_{\lambda(c)} = u_{1} u_{2} \cdots u_{n-1}$
where $u_{i} = \bv$ if $\lambda(x_{i-1},x_{i}) > \lambda(x_{i},x_{i+1})$
and $u_{i} = \av$ otherwise.

The bridge between posets and permutations is given by
the next result.
\begin{theorem}
For an $R$-labeling $\lambda$ of a graded poset $P$ we have that
$$   \Psi(P) = \sum_{c} u_{\lambda(c)}  ,  $$
where the sum is over the Jordan--H\"older set $\JH(P)$.
\end{theorem}
This is a reformulation of a result of
Bj\"orner and Stanley~\cite[Theorem~2.7]{Bjorner}.
The reformulation can be found
in~\cite[Lemma~3.1]{Ehrenborg_Readdy_r_cubical}.

As a corollary we obtain MacMahon's classical result on the major index
on a multiset; see~\cite{MacMahon_1}.
For a composition $\alpha$ of $n$ let
$\SSSS_{\alpha}$
denote all the permutations of the multiset
$\{1^{\alpha_{1}}, 2^{\alpha_{2}}, \ldots, k^{\alpha_{k}}\}$.
\begin{corollary}[MacMahon]
For a composition $\alpha = (\alpha_{1}, \alpha_{2}, \ldots, \alpha_{k})$
of $n$ the following identity holds:
$$   \sum_{\pi \in \SSSS_{\alpha}}   q^{\maj(\pi)}
    =
       \frac{[n]!}{[\alpha_{1}]! \cdot [\alpha_{2}]! \cdots [\alpha_{k}]!} .  $$
\label{corollary_permutation}
\end{corollary}
\begin{proof}
Let $P_{i}$ denote the chain of rank $\alpha_{i}$
for $i=1, \ldots, k$.
Furthermore, label all the cover relations in $P_{i}$ with~$i$.
Let~$L$ denote the distributive lattice
$P_{1} \times P_{2} \times \cdots \times P_{k}$.
Furthermore, let $L$ inherit an $R$-labeling from its factors,
that is, if
$x=(x_{1},x_{2}, \ldots, x_{k}) \coveredby (y_{1}, y_{2}, \ldots, y_{k})=y$
let the label $\lambda(x,y)$
be the unique coordinate~$i$ such that $x_{i} \coveredby y_{i}$.
Observe that the Jordan--H\"older set of $L$ is $\SSSS_{\alpha}$.
Direct computation yields $\Psi(P_{i}) = \av^{\alpha_{i}-1}$,
so the Major MacMahon map is $\Theta(\Psi(P_{i})) = 1$.
Iterating Theorem~\ref{theorem_Cartesian_product}
evaluates the Major MacMahon map on $L$:
\begin{align*}
     \sum_{\pi \in \SSSS_{\alpha}} q^{\maj(\pi)}
    & =
       \Theta\left(
       \sum_{\pi \in \SSSS_{\alpha}} u(\pi)
       \right)
    =
       \Theta\left(\Psi(L)\right)
    =
       \Gaussian{n}{\alpha}    . 
\qedhere
\end{align*}
\end{proof}

\begin{figure}
\setlength{\unitlength}{0.7 mm}
\begin{center}
\begin{picture}(80,40)(0,0)

\newcommand{\dddddot}{\circle*{2.0}}
\put(40,0)\dddddot
\put(0,20)\dddddot
\put(20,20)\dddddot
\put(80,20)\dddddot
\put(40,40)\dddddot

\newcommand{\ddddddot}{\circle*{1.0}}
\put(45,20)\ddddddot
\put(50,20)\ddddddot
\put(55,20)\ddddddot

\linethickness{0.30mm}
\qbezier(40,0)(20,10)(0,20)
\qbezier(40,0)(30,10)(20,20)
\qbezier(40,0)(60,10)(80,20)
\qbezier(40,40)(20,30)(0,20)
\qbezier(40,40)(30,30)(20,20)
\qbezier(40,40)(60,30)(80,20)

\put(7,6){\footnotesize $(-1,i)$}
\put(32,9){\footnotesize $(2,i)$}
\put(60,6){\footnotesize $(r_{i},i)$}
\put(14,29){\footnotesize $0$}
\put(25,29){\footnotesize $0$}
\put(66,29){\footnotesize $0$}

\end{picture}
\end{center}
\caption{The poset $P_{i}$ with its $R$-labeling
used in the proof of Corollary~\ref{corollary_signed_permutation}.}
\label{figure_one}
\end{figure}
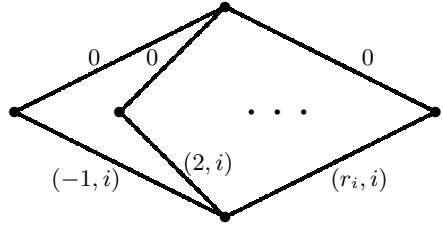

For a vector $\rv = (r_{1}, r_{2}, \ldots, r_{n})$
of positive integers
let an $\rv$-signed permutation be
a list
$\sigma = (\sigma_{1}, \sigma_{2}, \ldots, \sigma_{n+1}) = 
((j_{1},\pi_{1}),$ $(j_{2},\pi_{2}),$ $\ldots,$ $(j_{n},\pi_{n}),$ $0)$
such that
$\pi_{1} \pi_{2} \cdots \pi_{n}$
is a permutation
in the symmetric group $\SSSS_{n}$
and the sign $j_{i}$ is
from the set $S_{\pi_{i}} = \{-1\} \cup \{2, \ldots, r_{\pi_{i}}\}$.
On the set of labels
$\Lambda
  =
\{ (j,i) \: : \: 1 \leq i \leq n, j  \in S_{i} \} \cup \{0\}$
we use the lexicographic order with the
extra condition that $0 < (j,i)$ if and only if $0 < j$.
Denote the set of $\rv$-signed permutations
by $\SSSS_{n}^{\rv}$.
The descent set of an $\rv$-signed permutation~$\sigma$ is the set
$\Des(\sigma) = \{i \: : \: \sigma_{i} > \sigma_{i+1}\}$
and the major index is defined as
$\maj(\sigma) = \sum_{i \in \Des(\sigma)} i$.
Similar to Corollary~\ref{corollary_permutation},
we have the following result.
\begin{corollary}
The distribution of the major index for $\rv$-signed permutations
is given by
$$   \sum_{\sigma \in \SSSS^{\rv}_{n}} q^{\maj(\sigma)}
     =
        [n]!
          \cdot
        \prod_{i=1}^{n} (1 + (r_{i}-1) \cdot q)   .  $$
\label{corollary_signed_permutation}
\end{corollary}
\begin{proof}
The proof is the same as Corollary~\ref{corollary_permutation}
except we replace the chains with the posets $P_{i}$
in Figure~\ref{figure_one}.
Note that $\Psi(P_{i}) = \av + (r_{i} - 1) \cdot \bv$.
Let $L$ be the lattice
$L = P_{1} \diamond^{*} P_{2} \diamond^{*} \cdots \diamond^{*} P_{n}$.
Let~$L$ inherit the labels of the cover relations from its factors
with the extra condition that the cover relations attached to
the maximal element receive the label $0$.
This is an $R$-labeling and the labels of the maximal chains
are exactly the $\rv$-signed permutations.
\end{proof}

For signed permutations, that is,
$\rv = (2,2, \ldots, 2)$, the above result follows
from an identity due to Reiner~\cite[Equation~(5)]{Reiner}.

\section{Concluding remarks}

We suggest the following $q,t$-extension of
the Major MacMahon map $\Theta$.
Define
$\Theta^{q,t} : \zab \longrightarrow \Zzz[q,t]$
by
\begin{equation}
\Theta^{q,t}(w) 
       =
\Theta(w) \cdot  {w \vrule_{\av = 1, \bv = t}}
       =
\prod_{i \: : \: u_{i} = \bv} q^{i} \cdot t    ,
\label{equation_q_t}
\end{equation}
for an $\ab$-monomial $w = u_{1} u_{2} \cdots u_{n}$.
Applying this map to the $\ab$-index
of the Boolean algebra yields
one of the four types of $q$-Eulerian polynomials:
$$
   \Theta^{q,t}(\Psi(B_{n}))
   =
    A_{n}^{\maj,\des}(q,t)
   =
     \sum_{\pi \in \SSSS_{n}} q^{\maj(\pi)} t^{\des(\pi)}   .  $$
The following identity has been attributed
to Carlitz~\cite{Carlitz}, but goes back to
MacMahon~\cite[Volume~2, Chapter IV, \S 462]{MacMahon_book},
\begin{equation}
\sum_{k \geq 0}
[k+1]^{n} \cdot t^{k}
  =
\frac{A_{n}^{\maj,\des}(q,t)}
{\prod_{j=0}^{n} (1 - t \cdot q^{j})} .
\label{equation_Carlitz_MacMahon}
\end{equation}
For recent work on the $q$-Eulerian polynomials,
see 
Shareshian and Wachs~\cite{Shareshian_Wachs}.
It is natural to ask if there is a poset approach
to identity~\eqref{equation_Carlitz_MacMahon}.

In the second half of Section~\ref{section_permutations},
before Corollary~\ref{corollary_signed_permutation},
we offer one way to define a major index for signed
permutations.
However, 
there are several different ways to extend
the major index to signed permutations.
Two of our favorites are~\cite{Adin_Roichman,Steingrimsson}.

\section*{Acknowledgements}

The authors thank 
the referee for his careful comments.
The first author was partially supported by
National Security Agency grant~H98230-13-1-0280.
This work was partially supported by a grant from the Simons Foundation (\#206001 to Margaret~Readdy).

\newcommand{\journal}[6]{{\sc #1,} #2, {\it #3} {\bf #4} (#5), #6.}
\newcommand{\preprint}[3]{{\sc #1,} #2, preprint #3.}
\newcommand{\book}[4]{{\sc #1,} #2, #3, #4.}

\end{document}